\documentclass[12pt]{amsart}
\usepackage{amssymb,latexsym,amsmath,enumerate,nicefrac}
\usepackage{amsaddr}
\usepackage[total={6.5in,9in}, top=1in, left=1in, headsep = 0in]{geometry}

\usepackage{fancyhdr}
\usepackage{color}
\definecolor{myurlcolor}{rgb}{0.6,0,0}
\definecolor{mycitecolor}{rgb}{0,0,0.8}
\definecolor{myrefcolor}{rgb}{0,0,0.8}
\usepackage[bookmarks=false,pagebackref=true]{hyperref}
\hypersetup{colorlinks,
linkcolor=myrefcolor,
citecolor=mycitecolor,
urlcolor=myurlcolor}

\renewcommand*{\backref}[1]{}

\usepackage[all]{xy}
\usepackage{tikz}
\usetikzlibrary{calc}
\usepackage{graphics}
\usepackage{dsfont}
\usepackage{amsfonts}
\usepackage{mathtools}

\usepackage{comment}

\allowdisplaybreaks

\usepackage{graphics}
\usepackage{dsfont}
\usepackage{amsmath}
\usepackage{amsfonts,amssymb}
\usepackage{mathrsfs}
\usepackage{mathtools}

\usepackage{tikz}
\usetikzlibrary{calc}

\newcommand{\x}{[x]}
\newcommand{\vol}{{\rm Vol}}
\newcommand{\Prob}{{\rm Prob}}
\newcommand{\Pb}{{\rm P}}
\newcommand{\Pbm}{{\rm P}^m}
\newcommand{\Qbm}{{\rm Q}^m}
\newcommand{\WM}{{\rm W}}

\newcommand{\BB}{{\mathds B}}
\renewcommand{\SS}{{\mathds S}}

\newcommand{\e}{{\rm e}}
\newcommand{\X}{\tilde{X}}
\newcommand{\Y}{\tilde{Y}}
\newcommand{\Z}{\tilde{Z}}
\newcommand{\s}{\tilde{S}}

\newcommand{\C}{\mathscr C}
\newcommand{\Ell}{\tau}
\renewcommand{\(}{\left(} 
\renewcommand{\)}{\right)}
\DeclarePairedDelimiter\floor{\lfloor}{\rfloor}

\newtheorem{theorem}{Theorem}[section]

\newtheorem{proposition}{Proposition}[section]
\theoremstyle{definition}
\newtheorem{definition}{Definition}[section]

\theoremstyle{remark}

\makeatletter
\def\namedlabel#1#2{\begingroup
    #2%
    \def\@currentlabel{#2}%
    \phantomsection\label{#1}\endgroup
}
\makeatother

\begin{document}
\title{$p$\,-Adic Brownian Motion is a Scaling Limit}
\author{David Weisbart}\address{Department of Mathematics\\University of California, Riverside}\email{weisbart@math.ucr.edu}

\begin{abstract} 
A $p$-adic Brownian motion is a continuous time stochastic process in a $p$-adic state space that has a Vladimirov operator as its infinitesimal generator.  The current work shows that any such process is the scaling limit of a discrete time random walk on a discrete group.  Earlier work required the exponent of the Vladimirov operator to be in $(1, \infty)$, and the convergence was the weak convergence of probability measures on the Skorohod space of paths on a compact time interval.  The current approach simplifies the earlier approach, allows for any positive exponent, eliminates the restriction to compact time intervals, and establishes some moment estimates for the discrete time processes that are of independent interest.
\end{abstract}

\maketitle

\tableofcontents

%

\thispagestyle{empty}

\section{Introduction}

The fundamental solution to the diffusion equation in the real setting induces the Wiener measure $\WM$ on $C([0,\infty)\colon\mathds R)$, the space of continuous real-valued functions on $[0,\infty)$.  The probability space for Brownian motion is the pair $(C([0,\infty)\colon\mathds R), \WM)$.  The continuous paths form a closed subset of $D([0,\infty)\colon\mathds R)$, the Skorohod space of real-valued c\`{a}dl\`{a}g functions on $[0,\infty)$ endowed with the Skorohod metric.  There is a sequence of probability measures $(\Pb^m)$ on $D([0,\infty)\colon\mathds R)$ so that each $\Pbm$ is the probability measure for a discrete time random walk on a grid in $\mathds R$, and $(\Pbm)$ converges weakly to $\WM$.  The random walks associated to $(\Pbm)$ arise from the scaling of a single discrete time process on a discrete group.  In this precise sense, Brownian motion in $\mathds R$ is a scaling limit.  There are direct analogs of Brownian motion for paths in $\mathds Q_p$, the $p$-adic numbers, that are functions of a real time parameter.  Do these real time $p$-adic Brownian motions also arise as scaling limits?  An earlier study partially answered this question \cite{BW}, but only treated Brownian motions with Vladimirov operators with exponents greater than 1 as their infinitesimal generators.  Furthermore, the convergence that this study established is, for any positive $T$, the weak convergence of probability measures on the Skorohod space $D([0,T]\colon\mathds Q_p)$ of $\mathds Q_p$-valued paths.  

The current work improves the earlier study \cite{BW}, simplifies the approach by reducing all calculations that involve scaled processes to calculations that involve a single discrete time stochastic process that has a discrete quotient group of $\mathds Q_p$ as its state space, and establishes moment estimates for the discrete time processes that are of independent interest.  It treats any Vladimirov operator with a positive exponent, and the convergence it establishes is the weak convergence of probability measures on $D([0,\infty)\colon\mathds Q_p)$.  The current approach should be robust enough to permit an extension to the very general type of non-Archimedean analogs of Brownian motion that Varadarajan introduced \cite{var97}. 

Authors have traditionally referred to the type of processes that the present article studies as $p$-adic ultrametric diffusion processes \cite{ABKO:JPA:2002, Bik:UAA:2010}.  However, such processes may potentially be more general than the processes that the present work studies.  Reference to the \emph{real time} $p$-adic Brownian motion as a $p$-adic Brownian motion agrees with the general practice of referring to a \emph{real time} analog of Brownian motion in a space other than $\mathds R$ as a Brownian motion in the space, but it introduces some ambiguity with respect to the parameter space for those familiar with $p$-adic mathematical physics.  In particular, Bikulov and Volovich introduced in their seminal work \cite{BV:IzvMath:1997} a very different stochastic process, parameterized by a $p$-adic time variable, that they also referred to as $p$-adic Brownian motion.  In some respects, such a process is the true analog of real Brownian motion.  To clarify, the $p$-adic Brownian motion that the current paper studies is a stochastic process with a real time parameter and a $p$-adic state space.

The two principle ideas that motivate the study of $p$-adic mathematical physics are the idea that non-Archimedean physical models could describe the observed ultrametricity in certain complex systems and the idea that spacetime could have a non-Archimedean structure at extremely small distance and time scales.  Ultrametric structures in spin glasses were already implicit in Parisi's early works \cite{Parisi:PRL:1979, Parisi:JPA:1980}.  Avetisov, Bikulov, and Kozyrev \cite{Avetisov_Bikulov_Kozyrev:JPA:1999} and Parisi and Sourlas \cite{Parisi_Sourlas:JPA:1999} independently proposed $p$-adic models to describe replica symmetry breaking.  The understanding that there are fundamental limitations on physical measurements below the Planck scale motivated Volovich to introduce the \emph{Volovich hypothesis} \cite[Chapter 6]{V2011}, the idea that physical theories that deal with such ultramicroscopic measurements should involve non-Archimedean local fields \cite{vol1}.  The review of Dragovich, Khrennikov, Kozyrev, and Volovich \cite{drag} and the more recent work of Khrennikov, Kozyrev, and Z\'{u}\~{n}iga-Galindo \cite{KKZ} both give excellent accounts of the development of $p$-adic mathematical physics and list a variety of areas of potential application for the results of the current paper.  A significant application of $p$-adic Brownian motion to physics that has experimental confirmation is the utilization by Avetisov, Bikulov, and Zubarev \cite{ABZ:ProcSteklov:2014} and by Bikulov and Zubarev \cite{BZ:Physica:2021} of the Vladimirov equation in the description of protein molecule dynamics.  The current work could also permit a study from a discrete perspective of the work of Khrennikov and Kochubei \cite{KK:JFAA:2018} on a $p$-adic analog of the porous medium equation, as well as its generalization by the previous two authors with Antoniouk \cite{AKK:JPDOA:2020}.  It would also be of interest to extend the finite approximation of non-Archimedian quantum systems \cite{BDW} to a discrete time setting.

Taibleson introduced the idea of a pseudo-differential operator in the context of a local field \cite{Taib}.  Saloff-Coste studied such operators \cite{SC1} and generalized them to the setting of local groups \cite{SC2}.  The seminal works of Vladimirov and Volovich on $p$-adic quantum systems \cite{VV89a, VV89b} and Vladimirov's works in which he studied the \emph{Vladimirov operator} and computed its spectrum \cite{Vlad88, Vlad90} dramatically increased interest in analogs of diffusion equations in the $p$-adic setting.  Following the work of Kochubei \cite{koch92} and Albeverio and Karwowski \cite{alb}, Varadarajan introduced the idea of a Brownian motion in a state space that is a finite dimensional vector space over a division ring which is finite dimensional over a local field of arbitrary characteristic \cite{var97}.  The current paper takes the perspective of Varadarajan as its starting point, but restricts to the setting of $p$-adic state spaces.

Section~\ref{Sec-BM2} presents some necessary background for the current work and introduces to a general readership a framework for studying scaling limits in both the real and $p$-adic settings.  Section~\ref{Sec-Primitive} introduces the state space of the primitive process and some properties of the primitive process and its refinements.  It establishes in Theorem~\ref{LOM:Theorem:PrimitiveMoments} upper bounds for low order moments of the primitive process.  Section~\ref{Sec-ScalingLimits-Revised} introduces a family of scaled processes in $\mathds Q_p$ and uses Theorem~\ref{LOM:Theorem:PrimitiveMoments} to determine uniform upper bounds for the scaled processes that imply that the family of measures associated to the scaled processes is uniformly tight.  Uniform tightness of the measures and the convergence of the finite dimensional distributions of the scaled processes imply the main result of the paper, Theorem~\ref{secScalingLimit:Theorem:MAIN}.

\section{Background and Framework}\label{Sec-BM2}

\subsection{Path spaces}

Take $I$ to be either $[0, \infty)$ or $\mathds N_0$, the natural numbers with $0$. For any Polish space $\mathcal S$, denote by $F(I\colon \mathcal S)$ the set of all paths in $\mathcal S$ with domain $I$.  Two spaces of paths are of particular importance, the set $C([0,\infty) \colon \mathcal S)$ of continuous paths in $\mathcal S$ with domain $[0,\infty)$ equipped with the topology of uniform convergence on compacta, and the \emph{Skorohod space} $D([0,\infty)\colon \mathcal S)$ of c\`{a}dl\`{a}g functions from $[0,\infty)$ to $\mathcal S$ equipped with the Skorohod metric \cite{bil1}.  Probability measures on the second space are the central objects of the current study.  Denote by $\Omega(I)$ any of the spaces $F(I\colon \mathcal S)$, $C([0,\infty) \colon \mathcal S)$, or $D([0,\infty)\colon \mathcal S)$ and by $\mathcal B(\mathcal S)$ the set of Borel subsets of $\mathcal S$.  

For any natural number $N$, an \emph{epoch} of length $N$ in $I$ is a strictly increasing finite sequence in $I$ with domain $\{0, \dots, N\}$ such that $e(0)$ is equal to $0$.  A \emph{route}, $U$, of length $N$ in $\mathcal S$ is a finite sequence from $\{0, \dots, N\}$ to $\mathcal B(\mathcal S)$.  A \emph{history}, $h$, of length $N$ for $\mathcal S$-valued paths with domain $I$ is a finite sequence of ordered pairs in $I\times \mathcal B(\mathcal S)$ with $N+1$ places and with the property that if \[h = \((t_i, U_i)\)_{i\in \{0, \dots, N\}},\] then $(t_i)$ is an epoch in $I$ and $(U_i)$ is a route in $\mathcal B(\mathcal S)$.  For any history $h$, denote by $e_h$, $\ell(h)$, and $U_h$ the associated epoch, length, and route of $h$, respectivly.  Take $H(I, \mathcal S)$ to be the set of all histories for $\mathcal S$-valued paths with domain $I$.  On specifying a path space $\Omega(I)$, take $\C$ to be defined for any $h$ in $H(I, \mathcal S)$ by \begin{align*}\C(h) &= \left\{\omega \in \Omega(I)\colon \forall i\in \mathds N_0\cap [0, \ell(h)], \; \omega(e_h(i)) \in U_h(i)\right\}.\end{align*}  For any history $h$, $\C(h)$ is a \emph{simple cylinder set}.  The set $\C(H(I, \mathcal S))$ is the \emph{set of all simple cylinder sets of $\Omega(I)$} and is a $\pi$-system, but not an algebra. The \emph{set of cylinder sets} is the $\sigma$-algebra generated by $\C(H(I, \mathcal S))$---its elements are the \emph{cylinder sets}.

For any $t$ in $I$, take $Y_t$ to be the function that acts on any path $\omega$ in $\Omega(I)$ by \begin{equation}\label{First:DefofY}Y_t(\omega) = \omega(t).\end{equation}  Take $Y$ to be the function \[Y\colon I\times \Omega(I) \to \mathcal S \quad \text{by}\quad Y(t, \omega) = \omega(t).\] For any probability measure $\Pb$ on the $\sigma$-algebra of cylinder sets of $\Omega(I)$, the pair $(\Omega(I), \Pb)$ is a \emph{path space with probability measure $\Pb$}.  For any $t$ in $I$, $(\Omega(I), \Pb, Y_t)$ is a random variable and $(\Omega(I), \Pb, Y)$ is a stochastic process.  The \emph{finite dimensional distributions} for $(\Omega(I), \Pb, Y)$ are the probabilities for the simple cylinder sets of $\Omega(I)$.  The convergence of stochastic processes that this paper investigates involves processes that have the same state and parameter spaces, but different probability measures.  The level of precision in referring to stochastic processes as triples facilitates reframing the study of the convergence of a sequence of stochastic processes as the study of the convergence of a sequence of probability measures on the same Polish space.


\subsection{Abstract processes and models}

The construction of a probability measure on a space of paths typically involves for each $N$ in $\mathds N_0$ and each epoch $e$ of length $N$ the construction of a premeasure on the $\pi$-system of Borel subsets of $\mathcal S^{N}$ that are products of Borel sets in each component.  If this family of premeasures satisfies the Kolomogorov consistency conditions, then the Kolmogorov extension theorem guarantees that there is a unique probability measure $\Pb$ on the cylinder sets of $F(I\colon \mathcal S)$ that agrees with the given premeasures \cite{bil1}.  It is helpful to initially view the premeasures as being the finite dimensional distributions of a stochastic process, even before it is known that there is a stochastic process with the given finite dimensional distributions.  After determining that there is such a stochastic process, verification of certain moment estimates for the process will guarantee that the process has a version in either $C(I\colon \mathcal S)$ or $D(I\colon \mathcal S)$ \cite{bil1, cent}.

Refer to the construction of the finite dimensional distributions for a stochastic process as the construction of the \emph{abstract stochastic process} $\tilde{Y}$.  Refer to the construction of the law for a random variable without actually specifying its domain as the construction of an \emph{abstract random variable} $\tilde{X}$.  A model for the abstract random variable $\tilde{X}$ is a random variable with the same law as $\tilde{X}$.  A model for the abstract stochastic process $\tilde{Y}$ is a stochastic process with the same finite dimensional distributions as $\tilde{Y}$.  To distinguish the probabilities associated with abstract random variables and stochastic processes from the probabilities associated with their models, denote by ${\rm Prob}(\tilde{X}\in A)$ the probability that $\tilde{X}$ takes a value in some Borel set $A$, and use an analogous notation for $\tilde{Y}$.


\subsection{Brownian motion in $\mathds R$}

For any $D$ in $(0, \infty)$, the equation \begin{equation}\label{Intro:DiffusionEquation}\frac{\partial u}{\partial t} = \frac{D}{2}\frac{\partial^2u}{{\partial x}^2}\end{equation} is the diffusion equation with diffusion constant $D$.  Its fundamental solution is the function $\rho$ that is given for any pair $(t,x)$ in $(0,\infty)\times \mathds R$ by \begin{equation}\rho(t,x) = \frac{1}{\sqrt{2\pi Dt}}\exp\left(-\tfrac{x^2}{2Dt}\right).\end{equation}  The set $\{\rho(t, \cdot)\colon t\in \mathds R_+\}$ forms a semigroup under the convolution operation, a \emph{convolution semigroup of probability density functions}.  There is a unique probability measure $\WM$ on $C([0, \infty)\colon\mathds R)$ so that for any history $h$, and with $x_0$ taken to be equal to $0$, \begin{align}\label{Framework:EQ:FormulaFDD}\WM(\mathcal C(h)) = \int_{U_h(1)}\cdots \int_{U_h(\ell(h))} \prod_{i\in\{1, \dots, \ell(h)\}}\rho(e_h(i)-e_h(i-1), x_i-x_{i-1})\,{\rm d}x_1\cdots{\rm d}x_{\ell(h)}\end{align} if $0$ is in $U_h(0)$ and is equal to $0$ otherwise.  In studying questions about the approximation of Brownian motion by discrete time random walks, it is helpful to view the measure $\WM$ as being a measure on $D([0, \infty)\colon\mathds R)$ that gives full measure to the closed subset $C([0, \infty)\colon\mathds R)$.


\subsection{The additive group $\mathds Q_p$}

See the books by Gouv\^{e}a \cite{Gou} and by Ramakrishnan and Valenza \cite{Ram} for further background on the $p$-adic numbers that this subsection summarizes. Refer to the book by Vladimirov, Volovich, and Zelenov \cite{vvz} for both an introduction to $p$-adic mathematical physics and for many helpful examples of integration over $\mathds Q_p$.

For any prime $p$, denote by $|\cdot|_p$ the $p$-adic absolute value on $\mathds Q$ and by $\mathds Q_p$ the field of $p$-adic numbers, the analytic completion of $\mathds Q$ with respect to the metric induced by $|\cdot|_p$. Denote by $\BB_k(x)$ and $\SS_k(x)$ the sets \[\BB_k(x) = \{z\in\mathds Q_p\colon |x-z|_p \leq p^k\} \quad \text{and}\quad \SS_k(x) = \{z\in\mathds Q_p\colon |x-z|_p = p^k\},\] and by $\mathds Z_{p}$ the \emph{ring of integers}, the set $\BB_0(0)$.  Take $\mu$ to be the unique Haar measure on the locally compact Hausdorff abelian group $(\mathds Q_{p}, +)$ that gives $\mathds Z_{p}$ unit measure.  Compress notation by henceforth suppressing $p$ in the notation for $|\cdot|_p$, and wherever else it is unambiguous to do so.

For any $x$ in $\mathds Q_{p}$, there is a unique function $a_x$ with \begin{equation}\label{Def:ax}a_x\colon\mathds Z\to\{0, 1, 2, \dots, p-1\}\quad {\rm and}\quad  x = \sum_{k\in\mathds Z}a_x(k)p^k.\end{equation}  Denote respectively by $\{x\}$ and $\chi$ the rational number and the additive character given by \[\{x\} = \sum_{k<0} a_{x}(k)p^k \quad {\rm and}\quad \chi(x) = e^{2\pi{\sqrt{-1}}\{x\}}.\]  The Fourier transform $\mathcal F$ and inverse Fourier transform $\mathcal F^{-1}$ are unitary operators on $L^2(\mathds Q_{p})$ that are given for any $f$ in $L^2(\mathds Q_p)$ by \[(\mathcal Ff)(x) = \int_{\mathds Q_{p}}\chi(-xy)f(y)\,{\rm d}\mu\!\(y\) \quad{\rm and}\quad (\mathcal F^{-1}f)(y) = \int_{\mathds Q_{p}}\chi(xy)f(x)\,{\rm d}\mu(x),\] where the integrals are improper integrals for functions that are in $L^2(\mathds Q_p)\smallsetminus L^1(\mathds Q_p)$.

\subsection{Brownian motion in $\mathds Q_p$}

For further background on parabolic equations over non-Archimedean fields and the theory of $p$-adic diffusion, refer to the books by Kochubei \cite{koch01} and Z\'{u}\~{n}iga-Galindo \cite{Zun1}.  The latter reference discusses diffusion processes of a more general type than those to be presently studied. 

Fix $b$ to be in $(0,\infty)$ and denote by $SB(\mathds Q_{p})$ the Schwarz-Bruhat space of compactly supported, locally constant, $\mathds C$-valued functions.  Take ${\mathcal M}$ to be the multiplication operator that acts on $SB(\mathds Q_{p})$ by \[({\mathcal M}f)(x) = |x|^bf(x).\] Denote by $\Delta^\prime_b$ the unique self-adjoint extension of the essentially self-adjoint operator that acts on any $f$ in $SB(\mathds Q_{p})$ by \[(\Delta^\prime_b f)(x) = \big(\mathcal F^{-1}{\mathcal M}\mathcal Ff\big)\!(x).\]  For any $\mathds C$-valued function $g$ with domain $(0, \infty)\times \mathds Q_{p}$ and any $t$ in $(0, \infty)$, take $g_t$ to be the function that is defined for every $x$ in $\mathds Q_p$ by \[g_t(x) = g(t,x).\]  Denote by $\mathcal D(\Delta_b)$ the set of all such $g$ so that for all $t$ in $(0, \infty)$, $g_t$ is in the domain of $\Delta^\prime_b$.  Take $\Delta_b$ to be the \emph{Vladimirov operator with exponent $b$} that acts on any $g$ in $\mathcal D(\Delta_b)$ by \[(\Delta_b g)(t,x) =  (\Delta_b g_t)(x).\]  Similarly extend the Fourier transform to act on $\mathds C$-valued functions on $(0, \infty)\times \mathds Q_{p}$ that are square integrable for each positive $t$ and the operator $\frac{{\rm d}}{{\rm d}t}$ to act on $\mathds C$-valued functions on $(0, \infty)\times \mathds Q_{p}$ that for any $x$ in $\mathds Q_p$ are differentiable in the $t$ variable.

Fix $\sigma$ to be in $(0, \infty)$.  The function $\rho$ that is for any $(t,x)$ in $(0, \infty)\times\mathds Q_p$ given by \begin{equation}\label{ScalLim:Equation:QpDiff}\rho(t,x) = \left(\mathcal F^{-1}\e^{-\sigma t|\cdot|^b}\right)\!(x)\end{equation} is the fundamental solution to the pseudo-differential equation \begin{equation}\label{padicDiffusionEquation}\dfrac{{\rm d}u}{{\rm d}t} = -\sigma\Delta_b u.\end{equation} A minor modification of Varadarajan's arguments \cite{var97} shows that \begin{equation}\label{ScalLim:Equation:Qppdf}\rho(t,x) = \sum_{k\in \mathds Z} p^k\left(\e^{-\sigma tp^{kb}} - \e^{-\sigma tp^{(k+1)b}}\right) \mathds 1_{\BB_{{-k}}}(x).\end{equation} The set $\{\rho(t,\cdot)\colon t\in (0, \infty)\}$ is a convolution semigroup of probability density functions that determines the probabilities of simple cylinder sets  in the same way as in the real case, by \eqref{Framework:EQ:FormulaFDD}.  These probabilities determine a probability measure $\Pb$ on $D([0,\infty)\colon \mathds Q_{p})$ that is concentrated on the paths that are initially at 0. The triple $(D([0,\infty)\colon \mathds Q_{p}), \Pb, Y)$ is a \emph{$p$-adic Brownian motion}.


\subsection{Construction of approximants}

Henceforth, take $m$ to be an index that varies in $\mathds N_0$ and $\mathds F$ to be either $\mathds R$ or $\mathds Q_p$.  For any discrete, infinite, additive abelian group $G$, take $\tilde{X}$ to be an abstract $G$-valued random variable and $(\tilde{X}_i)$ to be a sequence of independent abstract random variables with the same distribution as $\tilde{X}$.  Denote by $\tilde{S}_n$ the abstract random variable \begin{equation}\label{PrimProcAbs}\tilde{S}_n = \tilde{X}_0  +\tilde{X}_1 +\dots +\tilde{X}_n,\end{equation} where $\tilde{X}_0$ almost surely takes on the value $0$, the identity element of the group.  The abstract stochastic process $\s$ is a \emph{primitive process} with \emph{generator} $\tilde{X}$.  For any history $((0, U_0), (n_1, U_1), \dots, (n_N, U_N))$ of length $N$ with epoch in $\mathds N_0$ and route in $G$, define \begin{align}\label{PremeasureFromAbstract}&{\rm Prob}\!\(\big(\s_{0}\in U_0\big)\cap\cdots\cap\big(\s_{n_N}\in U_N\big)\)\notag\\& \qquad\qquad = \sum_{x_0\in U_0}\cdots\sum_{x_N\in U_N} {\rm Prob}\(\s_{0} = x_0\)\prod_{1\le i\le N}{\rm Prob}\(\s_{n_i-n_{i-1}} = x_{i}-x_{i-1}\).\end{align}  For each $n$, take $S_n$ to be the function \[S_n\colon F(\mathds N_0\colon G) \to G \quad \text{by}\quad S_n(\omega) = \omega(n).\]  The Kolmogorov Extension theorem guarantees the existence of a measure $\Pb^\ast$ on the cylinder sets of $F(\mathds N_0\colon G)$ so that the stochastic process $(F(\mathds N_0\colon G), \Pb^\ast, S)$, the \emph{primitive process}, is a model for $\s$.

\begin{definition}
For any positive null sequence $(\delta_m)$, a sequence $(\Gamma_m)$ of \emph{spatial embeddings} of $G$ into $\mathds F$ with \emph{spatial scale} $\delta$ is a sequence of injective functions from $G$ to $\mathds F$ so that for any $f$ in $\mathds F$, there is a $g$ in $G$ so that \[0<|\Gamma_m(g) - f|\leq\delta_m,\] and for any $g_1$ and $g_2$ in $G$, \[|\Gamma_m(g_1) - \Gamma_m(g_2)| < \delta_m \quad \text{implies that}\quad g_1 = g_2.\]
\end{definition}

\begin{definition}
For any positive null sequence $(\tau_m)$, a \emph{sequence of spatiotemporal embeddings} of $\mathds N_0\times G$ into $[0, \infty)\times\mathds F$ with time scale $\tau$ and sequence of spatial embeddings $\Gamma$ is a sequence $\iota$ with \[\iota_m\colon \mathds N_0\times G \to[0,\infty)\times \mathds F\quad {\rm by}\quad \iota_m(n,g) = \left(n\tau_m, \Gamma_m(g)\right).\]
\end{definition}

For any sequence $\iota$ of spatiotemporal embeddings of $\mathds N_0\times G$ into $[0,\infty)\times \mathds F$ and for any $t$ in $[0,\infty)$, take $\tilde{Y}^m$ to be the abstract stochastic process that is given by \[\tilde{Y}^m_t = \Gamma_m\big(\s_{\floor{\frac{t}{\tau_m}}}\big).\] In the cases to be investigated, bounds on moments guarantee that $\tilde{Y}^m$ has a model with sample paths in $D([0,\infty)\colon \mathds F)$ that almost surely take values in $\Gamma_m(G)$.  Denote by $\Pbm$ the measure on the cylinder sets of $D([0,\infty)\colon \mathds F)$ that gives rise to the finite dimensional distributions of $\Y^m$.  Define $\iota_mS$ to be the process $(D([0,\infty)\colon \mathds F), \Pbm, Y)$.  For appropriate choices of the primitive process and the sequence of spatiotemporal embeddings, the sequence of measures $(\Pbm)$ converges weakly to the probability measure $\Pb$ for a Brownian motion in $\mathds F$.  Such a sequence of approximating measures exists for any Brownian motion in $\mathds F$.  In this precise sense, Brownian motion in $\mathds F$ is a scaling limit of a primitive discrete time random walk on $G$.

\subsection{Classical example: Real Brownian motion}

The abstract random variable $\X$ that is given by the law \[\begin{cases}{\rm Prob}\big(\X = -1\big) = \frac{1}{2}&\mbox{}\\[.2em]{\rm Prob}\big(\X = 1\big) = \frac{1}{2}&\mbox{}\end{cases}\] generates the abstract $\mathds Z$-valued stochastic process $\s$, by \eqref{PrimProcAbs}.  Calculate the mean of $\s_n$, ${\mathds E}[\s_n]$, and the variance of $\s_n$, ${\rm Var}[\s_n]$, to verify that \[{\mathds E}[\s_n] = 0\quad {\rm and}\quad {\rm Var}[\s_n] = n.\]  

For any space scale $(\delta_m)$ and time scale $(\tau_m)$, define $\Gamma$ and $\iota$ for any $(n, z)$ in $N_0\times \mathds Z$ by \[\Gamma_m(z) = \delta_mz \quad{\rm and} \quad \iota_m(n,z) = (n\tau_m, \Gamma_m(z)).\]  The stochastic process $\(F([0,\infty)\colon \mathds R), \Pbm, Y\)$ is a model for $\tilde{Y}^m$.  Denote by ${\mathds E}^m$ the expected value with respect to $\Pbm$. Independence of the $\X_i$ imply that for any positive $t$, \[{\mathds E}^m\Big[\left|Y_{t_2}- Y_{t_1}\right|^2\Big] = {\mathds E}\bigg[\Big|\delta_m\tilde{S}_{\floor{\tfrac{t_2}{\tau_m}}} - \delta_m\tilde{S}_{\floor{\tfrac{t_1}{\tau_m}}}\Big|^2\bigg] = \delta^2_m\left(\floor{\tfrac{t_2}{\tau_m}} - \floor{\tfrac{t_1}{\tau_m}}\right).\] Independence of the increments of $Y$ imply Proposition~\ref{prop:real:centesta}.

\begin{proposition}\label{prop:real:centesta}
For any strictly increasing sequence $(t_1, t_2, t_3)$ in $[0,\infty)$, \[{\mathds E}^m\big[\left|Y_{t_2}- Y_{t_1}\right|^2\left|Y_{t_3} - Y_{t_2}\right|^2\big] \leq \frac{\delta_m^4}{\tau_m^2}\(t_3 - t_1\)^2.\]  
\end{proposition}

Proposition~\ref{prop:real:centesta} implies that the process has a version with sample paths in $D([0,\infty)\colon \mathds R)$ \cite{bil1, cent}. Denote once again by $\Pbm$ the probability measure for this version and take $\iota_mS$ to be the process $(D([0,\infty)\colon \mathds R), \Pbm, Y)$.  If there is a positive constant $D$ with \begin{equation}\label{eqn:dtreal}\frac{\delta_m^2}{\tau_m}\to D,\end{equation} then the estimate given by Proposition~\ref{prop:real:centesta} is uniform in $m$, which implies the uniform tightness of $(\Pbm)$. Uniform tightness and the convergence on the simple cylinder sets of $(\Pbm)$ to the Wiener measure $\WM$ with diffusion constant $D$ together imply the weak convergence of $(\Pbm)$ to $\WM$ \cite{bil1,cent}.  Although each $\Pbm$ gives full measure to the $\delta_m\mathds Z$-valued step functions whose jumps occur almost surely in $\tau_m\mathds N$, the measure $\WM$ is concentrated on the continuous functions.

\section{The Primitive Process and its Generator}\label{Sec-Primitive}

\subsection{The primitive $p$-adic state space and its refinements}

Take $G_m$ to be the quotient group $\mathds Q_p\slash p^m\mathds Z_p$ and $[\cdot]_m$ to be the quotient map from $\mathds Q_p$ onto $G_m$ that is given for each $x$ in $\mathds Q_p$ by \[[x]_m = x + p^m\mathds Z_p.\]  Define the absolute value $|\cdot|_m$ on $G_m$ by \[|[x]_m|_m = \begin{cases}|x| &\mbox{if }[x]_m\ne [0]_m\\0 &\mbox{if }[x]_m= [0]_m.\end{cases}\] For any $g$ in $G_m$, denote by $\BB^m_k(g)$ and $\SS^m_k(g)$ the sets \[\BB^m_k(g) = \{h\in G_m\colon |h - g|_m \leq p^k\}\quad {\rm and} \quad\SS^m_k(g) = \{h\in G_m\colon |h - g|_m = p^k\}.\] Suppress $g$ in the notation when $g$ is equal to $[0]$.

Denote by $G$ the group $G_0$ and by $[\cdot]$ the quotient map $[\cdot]_0$.  Take $\mu_m$ to be the counting measure on $G_m$ scaled by a factor of $p^{-m}$, so that for any $g$ the volume of $\BB^m_k(g)$ and $\SS^m_k(g)$ are respectively given by \[\vol(\BB^m_k(g)) = p^{-m}p^k\quad{\rm and} \quad \vol(\SS^m_k(g)) = p^{-m}p^k\big(1 - \tfrac{1}{p}\big).\] For any integrable function $f$ on $G_m$, \begin{equation}\label{Eq:PrimRef:IntandSumatm}\int_{G_m} f([x]_m) \;{\rm d}\mu_m([x]_m) = \sum_{[x]_m\in G_m} f([x]_m)p^{-m}.\end{equation}  Proposition~\ref{StateSpaces:Intmoverballs} follows from the observation that any function that is defined on $\mathds Q_p$ by \[x \mapsto f([x]_m)\] is locally constant with radius of local constancy equal to $p^{-m}$, and the fact that any ball in $\mathds Q_p$ with radius at least $p^{-m}$ is a disjoint union of balls of radius $p^{-m}$.

\begin{proposition}\label{StateSpaces:Intmoverballs}
For any $\mathds C$-valued integrable function $f$ on $G_m$ and any ball $B$ of radius at least $p^{-m}$ in $\mathds Q_p$,  \[\int_{[B]_m} f([x]_m) \,{\rm d}\mu_m([x]_m) = \int_B f([x]_m) \,{\rm d}\mu(x).\]
\end{proposition}

Take $\alpha_m$ to be the isomorphism from $G$ to $G_m$ that is defined for any $[x]$ in $G$ by \[\alpha_m(x + \mathds Z_p) = p^mx + p^m\mathds Z_p.\]  The group $p^{-m}\mathds Z_p$ is the Pontryagin dual of the group $G_m$ and any element of $G_m$ is $[x]_m$ for some $x$ in $\mathds Q_p$.  Use the canonical inclusion map that takes $p^{-m}\mathds Z_p$ into $\mathds Q_p$ to define the dual pairing $\langle\cdot, \cdot\rangle_m$, the function \[\langle\cdot, \cdot\rangle_m \colon p^{-m}\mathds Z_p \times G_m \to \mathds S^1\quad {\rm by}\quad \langle [x]_m, y\rangle_m = \chi(xy).\]  For any $([x]_m, w, y)$ in $G_m\times p^{m}\mathds Z_p \times p^{-m}\mathds Z_p$, the product $wy$ is in $\mathds Z_p$. The equality \begin{align*}\chi((x+w)y)= \chi(xy)\chi(w y) = \chi(xy)\end{align*} follows from the fact that $\chi$ is an additive rank $0$ character on $\mathds Q_p$, and it implies that the definition of the dual pairing is independent of the choice of the representative. The equalities %
\begin{align*}
\langle \alpha_m([x]), y\rangle_m = \langle [p^mx]_m, y\rangle_m = \chi(p^mxy) = \langle [x], p^my\rangle
\end{align*}
imply Proposition~\ref{charactonGmrelatestoG}.

\begin{proposition}\label{charactonGmrelatestoG}
For any $(x, y)$ in $\mathds Q_p\times p^{-m}\mathds Z_p$, \[\langle [x]_m, y\rangle_m = \langle [x], p^my\rangle.\]
\end{proposition}

Define the Fourier transform $\mathcal F_m$ that takes $L^2(G_m)$ to $L^2(p^{-m}\mathds Z_p)$ and the inverse Fourier transform $\mathcal F^{-1}_m$ that takes $L^2(p^{-m}\mathds Z_p)$ to $L^2(G_m)$ to be given for any $f$ in $L^2(G_m)$ and any $\tilde{f}$ in $L^2(p^{-m}\mathds Z_p)$ by \begin{align*}(\mathcal F_mf)(y) &= \int_{G_m}\langle -[x]_m,y\rangle_m f([x]_m)\,{\rm d}\mu_m([x]_m)\\\text{and} \quad  (\mathcal F_m^{-1}\tilde{f})\left([x]_m\right) &= \int_{p^{-m}\mathds Z_p}\langle [x]_m,y\rangle_m \tilde{f}(y)\,{\rm d}\mu(y).\end{align*}


\subsection{The generator of the primitive process}

Take $\X$ to be the $G_0$-valued abstract random variable with probability mass function $\rho_{\X}$ given by \begin{equation}\label{equation:PrimitiveRVLaw}\begin{cases} \rho_{\X}(\x) = \frac{p^b-1}{p^{ib}}\frac{1}{p^i-p^{i-1}} &\mbox{if } |\x| = p^i\\\rho_{\X}([0]) = 0,&\mbox{ }\end{cases} \quad\text{so that}\quad \begin{cases} \Prob(\X\in \SS^0_i) = \frac{p^b-1}{p^{ib}} &\mbox{}\\\Prob(\X=[0]) = 0.&\mbox{}\end{cases}\end{equation}  The law for $\X$ is different from the one previously introduced \cite{BW}.  Denote by $\phi_{\X}$ the Fourier transform of $\rho_{\X}$, the characteristic function of $\X$ that has domain equal to $\mathds Z_p$.  To compress notation, write \[\rho(i) = \frac{p^b-1}{p^{ib}}\cdot\frac{1}{p^i-p^{i-1}}.\] %
For any $y$ in $\mathds Z_p$, there is a $k$ in $\mathds N_0$ so that \begin{equation}\label{Secfour:Equation:PrimYk-sum-y}|y| = p^{-k},\end{equation} and so
\begin{align}\label{Secfour:Equation:PrimYk-sum}
\phi_{\X}(y) & = \int_G\langle -\x,y\rangle\rho_X(\x)\,{\rm d}\mu_0(\x)\notag\\
& = \sum_{i\in \mathds N} \rho(i) \int_{\SS^0_i} \langle -\x,y\rangle\,{\rm d}\mu_0(\x)\notag\\
& = \sum_{i\in \mathds N} \rho(i)\left\{\int_{\BB^0_i} \langle -\x, y\rangle\,{\rm d}\mu_0(\x) - \int_{\BB^0_{i-1}} \langle -\x,y\rangle{\rm d}\mu_0(\x)\right\}\notag\\
& = -\rho(1)\int_{\BB^0_0}\langle -\x,y\rangle\,{\rm d}\mu_0(\x) + \sum_{i\in \mathds N} (\rho(i) - \rho(i+1))\int_{\BB^0_i} \langle -\x,y\rangle\,{\rm d}\mu_0(\x)\notag\\
& = -\frac{p^b-1}{p^b}\cdot\frac{1}{p-1}{\mathds 1}_{\mathds Z_p}(y) + (p^b-1)\Big(\frac{p}{p-1}\Big)\Big(1-\frac{1}{p^{b+1}}\Big)\sum_{i\in \mathds N} \frac{1}{p^{ib}}{\mathds 1}_{p^{-i}}(y)\notag\\
& = -\frac{p^b-1}{p^b}\cdot\frac{1}{p-1} + (p^b-1)\Big(\frac{p}{p-1}\Big)\Big(1-\frac{1}{p^{b+1}}\Big)\sum_{i=1}^k \frac{1}{p^{ib}}.
%
%
%
\end{align}
Use \eqref{Secfour:Equation:PrimYk-sum-y} to simplify \eqref{Secfour:Equation:PrimYk-sum} and obtain Proposition~\ref{Prop:CharPrim}. 

\begin{proposition}\label{Prop:CharPrim}
For any $y$ in $\mathds Z_p$, \[\phi_{\X}(y) = 1 - \alpha |y|^b \quad {\rm where}\quad \alpha = \frac{p^{b+1}-1}{p^{b+1}-p^b}.\] 
\end{proposition}

 \begin{proposition}\label{Prim:Prop:Boundonalphaoverpb}
 The values $\frac{\alpha}{p^b}$ and $\alpha -1$ are both in $(0, 1)$.
 \end{proposition}
 
\begin{proof}
The Vladimirov exponent $b$ is positive and $p$ is at least 2, so \[\alpha -1 = \frac{1}{p-1}\cdot\frac{p^b-1}{p^b} < 1.\]

Multiply $2p^{2b+1}\ln(p)$ by both sides of the inequality \[\frac{p-1}{p} > \frac{1}{2p^b}\] and write both sides of the resulting inequality as derivatives to obtain the inequality \begin{align*}\frac{{\rm d}}{{\rm d}b}(p^{2b+1} - p^{2b}) > \frac{{\rm d}}{{\rm d}b}(p^{b+1} - 1).\end{align*} The equality of $p^{b+1}-1$ and $p^{2b+1}-p^{2b}$ when $b$ is equal to $0$ implies that for any positive $b$, \[\frac{\alpha}{p^b} = \frac{p^{b+1}-1}{p^{2b+1}-p^{2b}}<1.\]
\end{proof}


\subsection{The primitive process}
 
Take $\s$ to be the abstract stochastic process defined by \eqref{PrimProcAbs} and, for any $n$ in $\mathds N$, take $\rho^\ast(n, \cdot)$ to be the probability mass function for $\s_n$.

\begin{proposition}
For any natural number $n$ and $[x]$ in $G$, \begin{equation}\label{eq:pmfforprimitiveprocess}\rho^\ast(n, \x) = (1-\alpha)^n\mathds 1_{\BB^0_0}(\x) + \sum_{i\in \mathds N}\Big(\Big(1-\frac{\alpha}{p^{ib}}\Big)^n - \Big(1-\frac{\alpha}{p^{(i-1)b}}\Big)^n\Big)p^{-i}\mathds 1_{\BB^0_i}(\x).\end{equation}
\end{proposition}

\begin{proof}
The Fourier transform takes the convolution of functions to their product, so compute the inverse Fourier transform of the $n^{\rm th}$ power of $\phi_{\X}$ to obtain the equalities
\begin{align*}
\rho^\ast(n, \x) 
&= \sum_{i\in \mathds N_0}\int_{|y| = p^{-i}}\langle \x,y\rangle\big(1-\alpha|y|^b\big)^n\,{\rm d}y\\
&= \sum_{i\in \mathds N_0}\Big(1-\frac{\alpha}{p^{ib}}\Big)^n\int_{|y| = p^{-i}}\langle \x,y\rangle\,{\rm d}y\\
&= (1-\alpha)^n\bigg(\int_{|y|\leq 1}\langle \x,y\rangle\,{\rm d}y - \int_{|y|\leq p^{-1}}\langle \x,y\rangle\,{\rm d}y\bigg)\\&\hspace{1in} + \sum_{i\in \mathds N}\Big(1-\frac{\alpha}{p^{ib}}\Big)^n\bigg(\int_{|y|\leq p^{-i}}\langle \x,y\rangle\,{\rm d}y - \int_{|y|\leq p^{-(i+1)}}\langle \x,y\rangle\,{\rm d}y\bigg)\\
&= (1-\alpha)^n\int_{|y|\leq 1}\langle \x,y\rangle\,{\rm d}y\\&\hspace{1in} + \sum_{i\in \mathds N}\Big(\Big(1-\frac{\alpha}{p^{ib}}\Big)^n - \Big(1-\frac{\alpha}{p^{(i-1)b}}\Big)^n\Big)\int_{|y|\leq p^{-i}}\langle \x,y\rangle\,{\rm d}y\\
&= (1-\alpha)^n\mathds 1_{\BB^0_0}(\x) + \sum_{i\in \mathds N}\Big(\Big(1-\frac{\alpha}{p^{ib}}\Big)^n - \Big(1-\frac{\alpha}{p^{(i-1)b}}\Big)^n\Big)p^{-i}\mathds 1_{\BB^0_i}(\x).
\end{align*}
\end{proof}

The Kolmogorov extension theorem guarantees that $\s$ has a model with sample paths in $F(\mathds N_0\colon G)$, a stochastic process $(F(\mathds N_0\colon G), \Pb^\ast, S)$ with finite dimensional distributions given by \eqref{PremeasureFromAbstract}.  Take $\mathds E_\ast[\cdot]$ to be the expected value with respect to $\Pb^\ast$.


\subsection{Bounds for low order moments of the primitive process}\label{sec:PrimitiveOrderEstimates}

In the real setting, moment estimates come from a straightforward calculation of variance.  Determining sufficient estimates in the $\mathds Q_p$ setting, irrespective of the value of $b$, is more involved.

\begin{theorem}\label{LOM:Theorem:PrimitiveMoments}
There is a $K$ in $\mathds R$ so that for any $r$ in $(0, b)$ and any $n$ in $\mathds N$, 
\begin{equation}\label{eq:KeyMomentEstimatePrimitive}
{\mathds E}_\ast\big[|S_n|^r\big] <Kn^{\frac{r}{b}}.\end{equation} 
\end{theorem}

\begin{proof}
Use the explicit expression for $\rho^\ast(n,x)$ to obtain the equalities
\begin{align}\label{sec:PrimitiveOrderEstimates}
{\mathds E}_\ast\big[|S_n|^r\big] &= \int_{G}|[x]|^r\rho^\ast(n,[x])\,{\rm d}\mu_0([x])\notag\\
&= \sum_{i\in \mathds N}\Big(\Big(1-\frac{\alpha}{p^{ib}}\Big)^n - \Big(1-\frac{\alpha}{p^{(i-1)b}}\Big)^n\Big)\int_{G}|[x]|^rp^{-i}\mathds 1_{B^0_i}([x])\,{\rm d}\mu_0([x]).
\end{align}
An integral over $G$ is a countable sum of integrals over the $\SS^0_i$, namely%
\begin{align}\label{sec:PrimitiveOrderEstimatesIntegral}
\int_{G}|[x]|^rp^{-i}\mathds 1_{\BB^0_i}([x])\,{\rm d}\mu_0([x]) & = \bigg(\int_{\SS^0_i}|[x]|^r\,{\rm d}\mu_0([x]) + \cdots + \int_{\SS^0_1}|[x]|^r\,{\rm d}\mu_0([x]) + 0\bigg)p^{-i}\notag\\
& = \Big(p^r\big(p^1 - 1\big) + \cdots + p^{ir}\big(p^i - p^{i-1}\big)\Big)p^{-i} \notag\\& = \frac{p^{r+1}-p^r}{p^{r+1}-1}\big(p^{ir} - p^{-i}\big) < \big(p^{ir} - p^{-i}\big).
\end{align} %
Equalities \eqref{sec:PrimitiveOrderEstimates} and \eqref{sec:PrimitiveOrderEstimatesIntegral} together imply that 
\begin{align*}
{\mathds E}_\ast\big[\big|S_n\big|^r\big] & < \sum_{i\in\mathds N}\big(p^{ir}-p^{-i}\big)\Big(\Big(1-\frac{\alpha}{p^{ib}}\Big)^n - \Big(1-\frac{\alpha}{p^{(i-1)b}}\Big)^n\Big)\\
&=\sum_{i\in\mathds N}p^{ir}\Big(\Big(1-\frac{\alpha}{p^{ib}}\Big)^n - \Big(1-\frac{\alpha}{p^{(i-1)b}}\Big)^n\Big) - \sum_{i\in\mathds N}p^{-i}\Big(\Big(1-\frac{\alpha}{p^{ib}}\Big)^n - \Big(1-\frac{\alpha}{p^{(i-1)b}}\Big)^n\Big).
\end{align*}

For any natural number $i$ that is greater than 1, \begin{equation}\label{Sec-UBPrim:Eq:DifftonandInt}\Big(1-\frac{\alpha}{p^{ib}}\Big)^n - \Big(1-\frac{\alpha}{p^{(i-1)b}}\Big)^n = n\int_{1-\frac{\alpha}{p^{(i-1)b}}}^{1-\frac{\alpha}{p^{ib}}}s^{n-1}\,{\rm d}s.\end{equation} The integrand in \eqref{Sec-UBPrim:Eq:DifftonandInt} is increasing, and so
\begin{equation}\label{Sec-UBPrim:Eq:DifftonandIntb}
\Big(1-\frac{\alpha}{p^{ib}}\Big)^n - \Big(1-\frac{\alpha}{p^{(i-1)b}}\Big)^n <n\Big(1-\frac{\alpha}{p^{ib}}\Big)^{n-1}\frac{\alpha}{p^{(i-1)b}}\Big(1-\frac{1}{p^b}\Big).
\end{equation}
 Take $I(n)$ to be the quantity 
 \begin{equation*}
 I(n) = \sum_{i>1}p^{ir}\Big(\Big(1-\frac{\alpha}{p^{ib}}\Big)^n - \Big(1-\frac{\alpha}{p^{(i-1)b}}\Big)^n\Big).
 \end{equation*}
For any natural number $i$ that is greater than 1, 
 \begin{equation*}
p^{-i}\Big(\Big(1-\frac{\alpha}{p^{ib}}\Big)^n - \Big(1-\frac{\alpha}{p^{(i-1)b}}\Big)^n\Big)>0,
 \end{equation*}
 and so
\begin{equation}\label{PrimMom:ESNthreetermsb}
{\mathds E}_\ast\big[\big|S_n\big|^r\big] < I(n)  + \big(p^{r}-p^{-1}\big)\Big(\Big(1-\frac{\alpha}{p^{b}}\Big)^n - (1-\alpha)^n\Big).
\end{equation}

The inequality \eqref{Sec-UBPrim:Eq:DifftonandIntb} implies that 
\begin{align}\label{Sec-UBPrim:Eq:toreindex}
I(n) & < \sum_{i>1}p^{ir}n\Big(1-\frac{\alpha}{p^{ib}}\Big)^{n-1}\frac{\alpha}{p^{(i-1)b}}\Big(1-\frac{1}{p^b}\Big)\notag\\
& = n\alpha^{\frac{r}{b}}\sum_{i>1}\Big(\frac{\alpha}{p^{ib}}\Big)^{-\frac{r}{b}}\Big(1-\frac{\alpha}{p^{ib}}\Big)^{n-1}\frac{\alpha}{p^{(i-1)b}}\Big(1-\frac{1}{p^b}\Big).
\end{align}
Write \[x_i = \frac{\alpha}{p^{ib}}, \quad I_i = [x_i, x_{i-1}], \quad {\rm and}\quad \Delta I_i = x_{i-1} - x_{i}\] to obtain the equality \begin{equation}\label{Sec-UBPrim:SimpWithXiIi}n\alpha^{\frac{r}{b}}\sum_{i>1}\Big(\frac{\alpha}{p^{ib}}\Big)^{-\frac{r}{b}}\Big(1-\frac{\alpha}{p^{ib}}\Big)^{n-1}\frac{\alpha}{p^{(i-1)b}}\Big(1-\frac{1}{p^b}\Big) = n\alpha^{\frac{r}{b}}\sum_{i>1}x_i^{-\frac{r}{b}}(1-x_i)^{n-1}\Delta I_i.\end{equation}
Reindex the sum in \eqref{Sec-UBPrim:Eq:toreindex} and use the fact that $\Delta I_i$ is equal to $p^b\Delta I_{i+1}$ together with \eqref{Sec-UBPrim:SimpWithXiIi} to see that 
\begin{align}\label{Sec-UBPrim:InpreBetafirst}
I(n) & < n\alpha^{\frac{r}{b}}p^b\sum_{i>2}x_{i-1}^{-\frac{r}{b}}\Big(1-x_{i-1}\Big)^{n-1}\Delta I_i,
\end{align}
and so
\begin{align}\label{Sec-UBPrim:InpreBeta}
I(n)& <  n\alpha^{\frac{r}{b}}p^b\int_0^{\frac{\alpha}{p^{2b}}}x^{-\frac{r}{b}}(1-x)^{n-1}\Delta I_i
\end{align}
since the sum in \eqref{Sec-UBPrim:InpreBetafirst} is a lower Riemann sum approximation of the integral in \eqref{Sec-UBPrim:InpreBeta}.  The positivity on $(0,1)$ of the integrand in \eqref{Sec-UBPrim:InpreBeta} implies that
\begin{align}\label{Sec-UBPrim:InpreBetab}
I(n) & <  n\alpha^{\frac{r}{b}}p^b\int_0^{1}x^{-\frac{r}{b}}(1-x)^{n-1}\,{\rm d}x\notag\\ 
& = n\alpha^{\frac{r}{b}}p^bB\Big(\frac{b-r}{b}, n\Big) = n\alpha^{\frac{r}{b}}p^b\frac{\Gamma\big(\frac{b-r}{b}\big)\Gamma(n)}{\Gamma\big(\frac{b-r}{b}+n\big)},
\end{align}
where $B$ is the beta function and $\Gamma$ is the gamma function.

Wendel proved \cite{Wend} that for any positive $x$ and $a$ in $[0,1)$, \begin{equation*}
1 \leq \frac{x^a\Gamma(x)}{\Gamma(x+a)} \leq \Big(\frac{x}{x+a}\Big)^{1-a},\end{equation*} and so \begin{equation}\label{Wend:Gamma}\frac{\Gamma(n)}{\Gamma\big(\frac{b-r}{b}+n\big)}\leq n^{-\frac{b-r}{b}}\Big(\frac{n+\frac{b-r}{b}}{n}\Big)^{1-\frac{b-r}{b}} = n^{-\frac{b-r}{b}}\Big(\frac{n+\frac{b-r}{b}}{n}\Big)^{\frac{r}{b}} = e_1(n)n^{-\frac{b-r}{b}}\end{equation}
where for all $n$, \[e_1(n) < 2 \quad {\rm and}\quad \lim_{n\to \infty} e_1(n) =1.\] The inequalities \eqref{Sec-UBPrim:InpreBetab} and \eqref{Wend:Gamma} together imply that
\begin{align}\label{Sec-UBPrim:InBetaalmostthere}
I(n) &< 
2\alpha^{\frac{r}{b}}p^b\Gamma\Big(\frac{b-r}{b}\Big)n^{\frac{r}{b}}.
\end{align}

To bound from above the second summand in  \eqref{PrimMom:ESNthreetermsb}, take \[\delta = \max\Big\{\Big|1-\frac{\alpha}{p^b}\Big|, |1-\alpha|\Big\} < 1\] so that \begin{equation}\label{PrimMom:equation:seconderrorterm} 0 < \bigg|\Big(1-\frac{\alpha}{p^b}\Big)^n - (1-\alpha)^n\bigg| \leq 2\delta^n = e_2(n)n^{\frac{r}{b}} \quad {\rm where}\quad e_2(n) = (2\delta^n n^{-\frac{r}{b}}).\end{equation} Since $e_2(n)n^{\frac{r}{b}}$ tends to zero as $n$ tends to infinity and $e_2(n)$ is bounded above by 2, the inequalities \eqref{PrimMom:ESNthreetermsb}, \eqref{Sec-UBPrim:InBetaalmostthere}, and \eqref{PrimMom:equation:seconderrorterm} together imply that
\begin{align*}
{\mathds E}\big[\big|S_n\big|^r\big] &< 2\alpha^{\frac{r}{b}}p^b\Gamma\Big(\frac{b-r}{b}\Big)n^{\frac{r}{b}} +2p^rn^{\frac{r}{b}} = Kn^{\frac{r}{b}}, \quad {\rm where}\quad K = 2\Big(p^r+\alpha^{\frac{r}{b}}p^b\Gamma\Big(\frac{b-r}{b}\Big)\Big).
\end{align*}
\end{proof}

\section{Convergence of the Discrete Time Processes}\label{Sec-ScalingLimits-Revised}

\subsection{Embeddings into the continuous state space}

Take $q_m$ to be an injection from $G_m$ to $\mathds Q_p$ that has the property that for any $g$ in $G_m$, $[q_m(g)]_m$ is equal to $g$.  Denote by ${\mathds 1}_m$ the indicator function on the set $(-\infty, m]\cap \mathds Z$.  Use $q_m$ and \eqref{Def:ax} to define for any $g$ in $G_m$ the injection $j_m$ from $G_m$ to $\mathds Q$ by \[j_m(g) =  \sum_{k\in \mathds Z} a_{q_m(g)}\!\(k\){\mathds 1}_m(k)p^k,\] which is independent of the choice of injection $q_m$.  The sequence of functions \begin{equation}\label{StateSpaces:Def:Gamma}\Gamma_m \colon G \to \mathds Q_p\quad {\rm by}\quad \Gamma_m(g) = j_m(\alpha_m(g))\end{equation} is a sequence of spatial embeddings of $G$ into $\mathds Q_p$ with spatial scale $(p^{-m})$.  For any $g$ in $G$, \begin{equation}\label{3:eqn:cosmetic}j_m(\alpha_m(g)) = p^mj_0(g).\end{equation} The sequence of discrete sets $(\Gamma_m(G))$ is a sequence of refinements of grids that approximates $\mathds Q_p$ in the sense that, for any $m_1$ and $m_2$ in $\mathds N_0$, if $m_1$ is less than $m_2$, then $\Gamma_{m_1}(G)$ is a subset of $\Gamma_{m_2}(G)$, and the union over all $m$ of the $\Gamma_m(G)$ is dense in $\mathds Q_p$. Equation~\eqref{3:eqn:cosmetic} implies Proposition~\ref{Prop:GammaAbsGm}.

\begin{proposition}\label{Prop:GammaAbsGm}
For any finite sequence $(g_i)_{i\in\{1, \dots, k\}}$ in $G$, \[\Big|\Gamma_m\!\Big(\sum_{i=1}^kg_i\Big)\Big| = p^{-m}\Big|\sum_{i=1}^{k}g_i\Big|.\]
\end{proposition}


\subsection{Spatiotemporal embeddings of the primitive process}

Denote by $\X^m$ the abstract random variable $\alpha_m\circ \X$.  The probability density function for $\X^m$, $\rho_{\X^m}$, is given for any $x$ in $\mathds Q_p$ by \[\rho_{\X^m}([x]_m) = \rho_{\X}([x]).\] The group $p^{-m}\mathds Z_p$ is the Pontryagin dual of $G_m$ and is the domain of $\phi_{\X^m}$, the characteristic function of $\X^m$.

\begin{proposition}\label{ScalLim:Prop:FTmprocess}
For any $y$ in $p^{-m}\mathds Z_p$, \[\phi_{\X^m}(y) = 1 - \alpha \frac{|y|^b}{p^{mb}}.\]
\end{proposition}

\begin{proof}
Propostion~\ref{StateSpaces:Intmoverballs} implies that for any $y$ in $p^{-m}\mathds Z_p$, 
\begin{align}\label{SEPP:EQ:NFirst}
\phi_{\X^m}(y) & = \int_{G_m}\langle-[x]_m, y\rangle_m\rho_{\X^m}([x]_m)\,{\rm d}\mu_m([x]_m)\notag\\
& = \int_{\mathds Q_p} \langle -[x]_m, y\rangle_m\rho_{\X^{m}}([x]_m)  \,{\rm d}\mu(x).
\end{align}
Proposition~\ref{charactonGmrelatestoG} and \eqref{SEPP:EQ:NFirst} together imply that
\begin{align}\label{SEPP:EQ:NSecond}
\phi_{\X^m}(y) & = \int_{\mathds Q_p} \langle -[x], p^my\rangle\rho_{\X^{m}}([x]_m)  \,{\rm d}\mu(x)\notag\\
& = \int_{\mathds Q_p} \langle -[x], p^my\rangle\rho_{\X}([x])  \,{\rm d}\mu(x).
\end{align}
Propostion~\ref{StateSpaces:Intmoverballs} permits the integral in \eqref{SEPP:EQ:NSecond} to be rewritten as an integral over $G$, and so
\begin{align*}
\phi_{\X^m}(y) & = \int_{G} \langle -[x], p^my\rangle\rho_{\X}([x])  \,{\rm d}\mu_0([x])\\
& = \Big(\mathcal F \rho_{\X}\Big)(p^my) = 1 - \alpha|p^my|^b = 1 - \frac{\alpha|y|^b}{p^{mb}},
\end{align*}
where Proposition~\ref{Prop:CharPrim} implies the penultimate equality.
\end{proof}

\renewcommand{\Ell}{\lambda}

Take $\tau$ to be the sequence of time scalings $\(\frac{\sigma}{\alpha}p^{mb}\right)$ and $\iota$ to be the sequence of spatiotemporal embeddings of $G$ into $\mathds Q_p$ with time scaling $\tau$ and spatial embeddings given by the $\Gamma_m$ that \eqref{StateSpaces:Def:Gamma} defines.  As discussed in the previous work \cite{BW}, the relationship between the space and time scalings that guarantee the convergence of the discrete time processes to the Brownian motion in $\mathds Q_p$ is similar in the $p$-adic and real settings, namely, \[\dfrac{\delta_m^b}{\tau_m} \to \frac{\sigma}{\alpha}.\] As should be expected, this relationship continues to hold for more general $b$.   Take $\Ell(m)$ to be the reciprocal of $\tau(m)$ and take $\Y^m$ and $\Z^m$ to be the abstract processes that are for any $t$ in $[0,\infty)$ given by \[\Y^m_t = \Gamma_m(\s_{\floor{t\Ell(m)}})\quad {\rm and} \quad \Z^m_t = \alpha_m(\s_{\floor{t\Ell(m)}}).\]  All probabilities for $\Y^m$ may be recovered from the probabilities for $\Z^m$, since \[\Y^m = j_m\circ \alpha_m\circ\s.\]  For all $t$ in $(0, \infty)$, denote by $\rho_{\Z^m}(t, \cdot)$ and $\phi_{\Z^m}(t, \cdot)$ the probability density function and characteristic function for $\Z^m_t$, respectively. %
Since the Fourier transform takes convolutions to products, Proposition~\ref{ScalLim:Prop:FTmprocess} implies Proposition~\ref{ScalLim:Prop:FTmprocessb}.

\begin{proposition}\label{ScalLim:Prop:FTmprocessb}
For any $y$ in $p^{-m}\mathds Z_p$ and $t$ in $[0,\infty)$, \[\phi_{\Z^m}(t,y) = \left(1 - \frac{\alpha|y|^b}{p^{mb}}\right)^{\floor{t\Ell(m)}}.\]
\end{proposition}

Take $E_{m}$ to be the function that is defined for any $(t,y)$ in $[0,\infty)\times\mathds Q_p$ by \[E_m(t, y) = \begin{cases}  
\left(1 - \frac{\alpha|y|^b}{p^{mb}}\right)^{\floor{t\Ell(m)}} &\mbox{if }|y| \leq p^{m}\\0 &\mbox{if }|y| > p^{m}.\end{cases}\]  For any fixed $b$ in $(0, \infty)$, Proposition~\ref{Prim:Prop:Boundonalphaoverpb} implies Proposition~\ref{last:Emlimit}.

\begin{proposition}\label{last:Emlimit}
For any $t$ in $[0, \infty)$, the sequence of functions $(E_m(t, \cdot))$ converges uniformly to ${\rm e}^{-t\sigma|\cdot|^b}$.
\end{proposition}


\subsection{Moment estimates for the embedded processes}

The Kolmogorov Extension Theorem guarantees the existence of a model $(F([0,\infty)\colon\mathds Q_p), \Qbm, Y)$ for $\Y^m$.  Denote by ${\mathds E}_m[\cdot]$ the expected value with respect to $\Qbm$.

\begin{proposition}\label{MainEstimateEmbedded}
There is a constant $C$ such that for any $t$ in $[0, \infty)$, \[{\mathds E}_m\!\left[\big|Y_t\big|^r\right] < Ct^{\frac{r}{b}}.\]
\end{proposition}

\begin{proof}
For any $t$ small enough so that $t\Ell(m)$ is in $[0,1)$, ${\mathds E}_m\!\left[\big|Y_t\big|^r\right]$ is equal to $0$, which verifies the inequality.  Denote by $K$ the constant appearing in Proposition~\ref{LOM:Theorem:PrimitiveMoments}.  If $t\Ell(m)$ is in $[1, \infty)$, then
\eqref{3:eqn:cosmetic} implies that %
\begin{align*}
{\mathds E}_m\!\left[\big|Y_t\big|^r\right] &= \sum_{k\in \mathds N_0}p^{kr}{\rm Prob}\(\big|\alpha_m\circ\s_{\floor*{t\Ell(m)}}\big|=p^k\)\\ &= \sum_{k\in \mathds N_0}p^{(k-m)r}{\rm Prob}\(\big|\s_{\floor*{t\Ell(m)}}\big|=p^k\)\\& = p^{-rm}{\mathds E}_\ast\!\left[\big|\s_{\floor*{t\Ell(m)}}\big|^r\right] \\&< Kp^{-rm}\floor*{t\Ell(m)}^{\frac{r}{b}} \leq K\(\tfrac{\sigma}{\alpha}\)^{\frac{r}{b}}  t^{\frac{r}{b}},
\end{align*} 
where Proposition~\ref{LOM:Theorem:PrimitiveMoments} implies the penultimate inequality.
\end{proof}

\begin{proposition}\label{tight}
There is a sequence of measures $(\Pbm)$ on $D([0,\infty)\colon \mathds Q_p)$ so that for each $m$ in $\mathds N_0$, the triple $(D([0,\infty)\colon \mathds Q_p), \Pbm, Y)$ is a model for $\Y^m$.  Furthermore, the sequence $(\Pbm)$ is a sequence of uniformly tight probability measures.
\end{proposition}

\begin{proof}
 For any strictly increasing finite sequence $(t_1, t_2, t_3)$ in $[0, \infty)$, the independence of the increments of $\Y^m$ implies that 
\begin{align*}
{\mathds E}_m\!\left[\big|Y_{t_3} - Y_{t_2}\big|^r\big|Y_{t_2} - Y_{t_1}\big|^r\right] & = {\mathds E}_m\!\left[\big|Y_{t_3} - Y_{t_2}\big|^r\right]{\mathds E}_m\!\left[\big|Y_{t_2} - Y_{t_1}\big|^r\right]\\ & \leq C(t_3-t_2)^\frac{r}{b}C(t_2-t_1)^\frac{r}{b} \leq C^2(t_3-t_1)^{\frac{2r}{b}}.
\end{align*}
Take $r$ to be in $\left(\frac{b}{2}, b\right)$ to verify that $(F([0,\infty)\colon\mathds Q_p), \Qbm, Y)$ satisfies Chentsov's criterion with constant $C$ independent of both $m$ and the choice of $(t_1, t_2, t_3)$ \cite{cent}.  
\end{proof}


\subsection{Convergence of the processes}

The sequence $(D([0,\infty)\colon \mathds Q_p), \Pbm, Y)$ of stochastic processes describe random walks on the embedded grids $\Gamma_m(G)$ in $\mathds Q_p$, with jumps that occur only at time points that are positive multiples of $\tau_m$.  Proposition~\ref{prop:qp:concenration} makes this idea precise.

\begin{proposition}\label{prop:qp:concenration}
The measure $\Pbm$ is concentrated on the subset of paths in $D([0, \infty)\colon \mathds Q_p)$ that are valued in $\Gamma_m(G)$ and are, for each natural number $n$, constant on the intervals $\big[(n-1)\tau_m, n\tau_m\big)$.
\end{proposition}

\begin{proof}
For any strictly increasing finite sequence $(t_1, \dots, t_k)$ in $[0,\infty)$, paths in $D([0, \infty)\colon \mathds Q_p)$ are $\Pbm$-almost surely in $\Gamma_m(G)$ at any given place of $(t_1, \dots, t_k)$.  Finite intersections of almost sure events are almost sure and so \[\Pbm\big(\{\omega \in D([0,\infty)\colon \mathds Q_p)\colon \omega(t_i) \in \Gamma_m(G), 1 \leq i \leq k\}\big) = 1.\]  For any sequence $(A_i)$ of strictly increasing nested finite subsets of $\mathds Q$ whose union is $\mathds{Q}$, continuity from above of $\Pbm$ implies that %
\begin{align*}
&\Pbm\big(\{\omega\in D([0,\infty)\colon \mathds Q_p)\colon \omega(s) \in \Gamma_m(G), \forall s \in \mathds{Q}\}\big)\\ &\qquad = \lim_{i \rightarrow \infty} \Pbm\big(\{\omega\in D([0,\infty)\colon \mathds R)\colon \omega(s) \in \Gamma_m(G), \forall s \in A_i\}\big) = 1.
\end{align*}%
The right continuity of the paths implies that the set of $\Gamma_m(G)$-valued paths has full measure with respect to the measure $\Pbm$.  

For any interval $I$ in $[0,\infty)$ that does not intersect the set $\tau_m\mathds N$, take $(V_i)$ to be a nested sequence of finite subsets of $I$ whose union is $I\cap \mathds Q$.  For any pair of real numbers $s_1$ and $s_2$ in $I$, $\floor*{s_1\Ell(m)}$ is equal to $\floor*{s_2\Ell(m)}$ and so \begin{align*} \Pbm\big(Y_{s_1} - Y_{s_2} = 0\big) & = {\rm Prob}\big(\Y_{s_1}^m - \Y_{s_2}^m = 0\big) \\&= {\rm Prob}\big(\s_{\floor*{s_1\Ell(m)}} - \s_{\floor*{s_2\Ell(m)}} = 0\big) = 1.\end{align*}  Continuity from above of the measure $\Pbm$ implies that for any $t$ in $I$, \begin{align*}&\Pbm\big(\left\{\omega\in D([0,\infty)\colon \mathds Q_p)\colon |\omega(t) - \omega(s)| = 0, \quad \forall s \in I\cap \mathds Q\right\}\big) \\&= \lim_{i\to \infty} \Pbm\big(\left\{\omega\in D([0,\infty)\colon \mathds Q_p)\colon |\omega(t) - \omega(s)| = 0, \quad \forall s \in I\cap V_i\right\}\big) = 1.\end{align*}  The right continuity of the paths implies that \[\Pbm\big(\left\{\omega\in D([0,\infty)\colon \mathds Q_p)\colon |\omega(t) - \omega(s)| = 0, \quad \forall s, t \in I\right\}\big) = 1.\]  Take $(I_n)$ to be a sequence of disjoint intervals that do not intersect $\tau_m\mathds N$ and whose union is $\mathds Q\setminus \tau_m\mathds N$.  Continuity from above of the measure $\Pbm$ implies that \[\Pbm\Big(\bigcap_{n\in \mathds N}\left\{\omega\in D([0,\infty)\colon \mathds Q_p)\colon |\omega(t) - \omega(s)| = 0, \quad \forall s,t \in I_n\right\}\Big) = 1.\]%
\end{proof}

\begin{definition}
The set $H_R$ of all \emph{restricted histories} for paths in $\mathds Q_p$ is the set of all histories in $\mathds Q_p$ whose route is a finite sequence of balls.
\end{definition}

\begin{proposition}\label{5:prop:restrictedhistconv}
For any restricted history $h$ in $H_R$, \[\Pbm(\C(h)) \to \Pb(\C(h)).\]
\end{proposition}

\begin{proof} %
For any $(t,x)$ in $(0,\infty)\times \mathds Q_p$,
\begin{align*}
 \rho_{\Z^m}(t, [x]_m) &= \int_{p^{-m}\mathds Z_p} \langle [x]_m,y\rangle \left(1 - \frac{\alpha|y|^b}{p^{mb}}\right)^{\floor{t\ell(m)}}{\rm d}\mu(y)
\\ & = \int_{\mathds Q_p} \chi(xy) E_m(t,y)\,{\rm d}\mu(y).
\end{align*} %
Denote by $F_m$ the function that takes any $x$ in $\mathds Q_p$ to $E_m(t,x)\mathds 1_{\BB_{m-1}(0)}(x)$, and denote by $A_m$ the integral of $|E_m(t,\cdot)|$ over $\SS_m(0)$. The bounds given by Proposition~\ref{Prim:Prop:Boundonalphaoverpb} imply that $(A_m)$ is a positive null sequence and that the sequence of functions $(F_m)$ increases to ${\rm e}^{-t\sigma|\cdot|^b}$.  There is, therefore, a positive null sequence $(\varepsilon_k)$ so that for any $m$, \begin{equation}\label{ConPross:Eq:ResHis:seq}\int_{\BB_k(0)^c}|E_m(t, \cdot)|\,{\rm d}\mu(x) < \varepsilon_k.\end{equation} Proposition~\ref{last:Emlimit} and \eqref{ConPross:Eq:ResHis:seq} together imply that for any positive $t$, 
 \begin{align*}
\left|\rho_{\Z^m}(t, [x]_m) - \rho_Y(t, x)\right| & \leq \int_{\mathds Q_p} \left|\chi(xy) E_m(t,y) - \chi(xy){\rm e}^{-\sigma t|y|^b}\right|\,{\rm d}\mu(y)\\& = \int_{\mathds Q_p} \left|E_m(t,y) - {\rm e}^{-\sigma t|y|^b}\right|\,{\rm d}\mu(y) \to 0,
\end{align*}
and so $(\rho_{\Z^m}(t, [\cdot]_m))$ converges uniformly on $\mathds Q_p$ to $\rho_{Y}(t, \cdot)$.

It is enough to suppose that $U_h(0)$ is the set $\{0\}$.  Simplify the notation by writing \[e_h = (t_0, \dots, t_k) \quad {\rm and}\quad U_h = (\{0\}, U_1, \dots, U_k).\]  As long as $m$ is large enough so that $p^{-m}$ is less than the radius of any ball in $U_h$,  
\begin{align*}
\Pbm(\C(h)) &= \Pbm(Y_{t_1}\in U_1, Y_{t_2}\in U_2, \dots, Y_{t_n}\in U_n)\\
&= \int_{U_1} \cdots \int_{U_k} \prod_{i\in\{1, \dots, k\}}\rho_{\Z^m}(t_i-t_{i-1}, [x_i] -[x_{i-1}])\,{\rm d}\mu(x_k)\cdots {\rm d}\mu(x_1)\\
&\to  \int_{U_1}\cdots \int_{U_k} \prod_{i\in\{1, \dots, k\}}\rho_{Y}(t_i-t_{i-1}, x_i-x_{i-1})\,{\rm d}\mu(x_k)\cdots {\rm d}\mu(x_1) = \Pb(\C(h)),
\end{align*}
since the integrand that is indexed by $m$ is uniformly convergent as a product of bounded, uniformly convergent functions.
\end{proof}

\begin{theorem}\label{secScalingLimit:Theorem:MAIN}
The sequence of measures $(\Pbm)$ converges weakly to $\Pb$.
\end{theorem}

\begin{proof}
Proposition~\ref{tight} implies that the family measures $\{\Pbm\colon m\in\mathds N_0\}$ is uniformly tight.  Since the intersection of any two balls in $\mathds Q_p$ is again a ball in $\mathds Q_p$, $\C(H_{R})$ is a $\pi$-system that generates the $\sigma$-algebra of cylinder sets.  Proposition~\ref{5:prop:restrictedhistconv} implies that for any $A$ in $\C(H_{R})$, $(\Pbm(A))$ converges to $\Pb(A)$.   The uniform tightness together with the convergence of the measures on a $\pi$-system that generates the cylinder sets implies the weak convergence of the $\Pbm$ to $\Pb$.
\end{proof}

The convergence demonstrated here is, for any $b$ in $(0, \infty)$, the weak convergence of probability measures on $D([0,\infty)\colon \mathds Q_p)$. 
In contrast, the convergence established earlier \cite{BW} was, for any fixed positive $T$ and $b$ in $(1, \infty)$, the weak convergence of probability measures on $D([0,T]\colon \mathds Q_p)$.  The current proof is quite robust and it should extend to the full generality of the non-Archimedean Brownian motion that Varadarajan introduced \cite{var97}.



\begin{thebibliography}{}

%
\bibitem{alb}  Albeverio, S., Karwowski, W.: \textsl{A random walk on $p$-adics - the generator and its spectrum.} Stochastic Processes and their Applications \textbf{53} l-22, (1994).

%
\bibitem{AKK:JPDOA:2020} Antoniouk, A.V., Khrennikov, A.Y., Kochubei, A.N.: \textsl{Multidimensional nonlinear pseudo-differential evolution equation with $p$-adic spatial variables}.  Journal of Pseudo-Differential Operators and Applications, 11, 311--343, (2020).

%
\bibitem{Avetisov_Bikulov_Kozyrev:JPA:1999}  Avetisov, V.A., Bikulov, A.H., Kozyrev, S.V.: \textsl{Application of $p$-adic analysis to models of spontaneous breaking of replica symmetry.} J. Phys. A: Math. Gen. 32 (50), 8785-8791, (1999).

%
\bibitem{ABKO:JPA:2002} Avetisov, V.A., Bikulov, A.H., Kozyrev, S.V., Osipov, V.A.: \textsl{$P$-Adic models of ultrametric diffusion constrained by hierarchical energy landscapes}. J. Phys. A: Math. Gen. 35 (2), 177, (2002).

%
\bibitem{ABZ:ProcSteklov:2014} Avetisov, V.A., Bikulov, A.K., Zubarev, A.P.: \textsl{Ultrametric random walk and dynamics of protein molecules}. Proc. Steklov Inst. Math. 285, 3-25, (2014). 

%
\bibitem{BDW}  Bakken, E. M., Digernes, T., Weisbart, D.: \textsl{Brownian motion and finite approximations of quantum systems over local fields}. Reviews in Mathematical Physics, vol. 29, no. 5, (2017).  

%
\bibitem{BW}  Bakken, E., Weisbart, D.:  \textsl{$p$\,-Adic brownian motion as a limit of discrete time random walks}. Commun. Math. Phys. 369, 371-402, (2019).

%
\bibitem{BV:IzvMath:1997} Bikulov, A.K., Volovich, I.V.: \textsl{$p$-Adic Brownian motion}. Izvestiya: Mathematics, 61 (3), 537, (1997).

%
\bibitem{Bik:UAA:2010} Bikulov, A.K.: \textsl{Problem of the first passage time for $p$-adic diffusion}. $P$-Adic Numbers, Ultrametric Analysis, and Applications, 2(2), 89-99, (2010).

%
\bibitem{BZ:Physica:2021} Bikulov, A.K., Zubarev, A.P.: \textsl{Ultrametric theory of conformational dynamics of protein molecules in a functional state and the description of experiments on the kinetics of CO binding to myoglobin}. Physica A: Statistical Mechanics and its Applications, 583, 126280, (2021).

%
\bibitem{bil1} Billingsley, P.: \textsl{Convergence of Probability Measures, Second Edition}. John Wiley $\&$ Sons, (1999).

%
\bibitem{cent} Chentsov, N.N.: \textsl{Weak convergence of stochastic processes whose trajectories have no discontinuities of the second kind and the ``heuristic'' approach to the Kolmogorov--Smirnov tests.} Theory of Probability \& Its Applications, 1(1):140-144, (1956).

%
\bibitem{drag} Dragovich, B., Khrennikov, A.Yu., Kozyrev, S.V., Volovich, I.V.: \textsl{On p-adic mathematical physics.} $p$-Adic Numbers, Ultrametric Analysis, and Applications. Volume 1, Issue 1, 1-17, (2009).

%
\bibitem{Gou} Gouv\^{e}a, F.Q.: \textsl{$p$-Adic Numbers.} Universitext. Springer, Berlin, Heidelberg, (1997).

%
\bibitem{KK:JFAA:2018} Khrennikov, A.Y., Kochubei, A.N.: \textsl{$p$-Adic analogue of the porous medium equation}.  J Fourier Anal Appl, 24(5), 1401-1424, (2018).

%
\bibitem {KKZ}  Khrennikov, A., Kozyrev, S., Z\'{u}\~{n}iga-Galindo, W.A.: \textsl{Ultrametric Equations and its Applications.} Encyclopedia of Mathematics and its Applications vol. 168, Cambridge University Press, (2018).

%
\bibitem{koch92} Kochubei, A.N.: \textsl{Parabolic equations over the field of $p$-adic numbers.} Math. USSR Izvestiya 39, 1263-1280, (1992).

%
\bibitem{koch01} Kochubei, A.N.: \textsl{Pseudo-Differential Equations and Stochastics over non-Archimedean Fields.} Monographs and Textbooks in Pure and Applied Mathematics 244 (Marcel Dekker Inc., New York, 2001).

%
\bibitem{Parisi:PRL:1979} Parisi, G.: \textsl{Infinite number of order parameters for spin-glasses.} Phys. Rev. Lett. 43, 1754, (1979).

%
\bibitem{Parisi:JPA:1980} Parisi, G.: \textsl{A sequence of approximate solutions to the S-K model for spin glasses.} J. Phys. A: Math. Gen. 13, (1980)

%
\bibitem{Parisi_Sourlas:JPA:1999} Parisi, G., Sourlas, N.: \textsl{$p$-Adic numbers and replica symmetry breaking.} Europ. Phys. J. B 14, 535-542, (2000)

%
\bibitem{Ram} Ramakrishnan, D., Valenza, R.J.: \textsl{Fourier Analysis on Number Fields.} Springer, New York, (1999).

%
\bibitem{SC1} Saloff-Coste, L.: \textsl{Op\'{e}rateurs pseudo-diff\'{e}rentiels sur un corps local.} C. R. Acad. Sci. Paris S\'{e}r. I 297, 171-174, (1983).

%
\bibitem{SC2} Saloff-Coste, L.: \textsl{Op\'{e}rateurs pseudo-diff\'{e}rentiels sur certains groupes totalement discontinus.} Studia Math. 83, 205-228, (1986).

%
\bibitem{Taib} Taibleson, M.H.: \textsl{Fourier Analysis on Local Fields.} Princeton University Press, Princeton, N.J.; University of Tokyo Press, Tokyo, (1975).

%
\bibitem{var97} Varadarajan, V.S.: \textsl{Path integrals for a class of $p$-adic Schr{\"o}dinger equations.} Lett. Math. Phys. 39, no. 2, 97-106, (1997).

%
\bibitem{V2011} Varadarajan, V.S.: \textsl{Reflections on Quanta, Symmetries, and Supersymmetries.} Springer, (2011).

%
\bibitem{Vlad88} Vladimirov, V.S.: \textsl{Generalized functions over the field of $p$-adic numbers.} Russian Math. Surveys, 43:5, 19-64, (1988).

%
\bibitem{Vlad90} Vladimirov, V.S.: \textsl{On the spectrum of some pseudo-differential operators over $p$-adic number field.} Algebra and analysis 2, 107-124, (1990).

%
\bibitem{VV89a} Vladimirov, V.S., Volovich, I.V.: \textsl{$p$-Adic quantum mechanics.} Comm. Math. Phys., 123:4, 659-676, (1989).

%
\bibitem{VV89b} Vladimirov, V.S., Volovich, I.V.: \textsl{$p$-Adic Schr\"{o}dinger-type equation.} Lett. Math. Phys., 18:1, 43-53, (1989).

%
\bibitem{vvz} Vladimirov, V.S., Volovich, I.V., Zelenov, E.I.: \textsl{$p$-Adic Analysis and Mathematical Physics.}  World Scientific, (1994).

%
\bibitem{vol1} Volovich, I.V.: \textsl{Number theory as the ultimate physical theory.} Preprint CERN-TH. 87, 4781-4786, (1987).

%
\bibitem{Wend} Wendel, J.G.: \textsl{Note on the gamma function.} Amer. Math. Monthly, 55:563-564, (1948).

%
\bibitem {Zun1}  Z\'{u}\~{n}iga-Galindo, W.A.: \textsl{Pseudodifferential Equations Over Non-Archimedean Spaces.} Lecture Notes in Mathematics (2174), Springer; 1st ed., (2016).




\end{thebibliography}
\end{document}